\documentclass[11pt]{amsart}
\usepackage{amssymb,amsmath,amsthm,newlfont,enumerate}

\theoremstyle{plain}
\newtheorem{theorem}{Theorem}[section]
\newtheorem{lemma}[theorem]{Lemma}
\newtheorem{corollary}[theorem]{Corollary}

\newtheorem{questions}[theorem]{Questions}

\newcommand{\CC}{{\mathbb C}}
\newcommand{\DD}{{\mathbb D}}

\newcommand{\cD}{{\mathcal D}}
\newcommand{\cH}{{\mathcal H}}
\newcommand{\cM}{{\mathcal M}}

\let\Re\undefined
\DeclareMathOperator{\Re}{Re}
\DeclareMathOperator{\hol}{\mathrm Hol}

\begin{document}

\title{A Gleason--Kahane--\.Zelazko  theorem for the Dirichlet space}

\author[J. Mashreghi]{Javad Mashreghi}
\address{D\'epartement de math\'ematiques et de statistique, Universit\'e Laval, 
Qu\'ebec City (Qu\'ebec),  Canada G1V 0A6.}
\email{javad.mashreghi@mat.ulaval.ca}

\author[J. Ransford]{Julian Ransford}
\address{D\'epartement de math\'ematiques et de statistique, Universit\'e Laval, 
Qu\'ebec City (Qu\'ebec),  Canada G1V 0A6.}
\email{julian.ransford.1@ulaval.ca}

\author[T. Ransford]{Thomas Ransford}
\address{D\'epartement de math\'ematiques et de statistique, Universit\'e Laval, 
Qu\'ebec City (Qu\'ebec),  Canada G1V 0A6.}
\email{thomas.ransford@mat.ulaval.ca}

\thanks{JM supported by an NSERC Discovery Grant. 
JR supported by an NSERC CGS-M Scholarship.
TR supported by an NSERC Discovery Grant and a Canada Research Chair}

\date{16 October 2017}

\begin{abstract}
We show that every linear functional on the Dirichlet space that is non-zero on nowhere-vanishing functions is necessarily a multiple of a point evaluation. Continuity of the functional is not assumed. As an application, we obtain a characterization of weighted composition operators on the Dirichlet space as being exactly those linear maps that send nowhere-vanishing functions to nowhere-vanishing functions.

We also investigate possible extensions to weighted Dirichlet spaces with superharmonic weights. As part of our investigation, we are led to determine which of these spaces contain functions that map the unit disk onto the whole complex plane.
\end{abstract}



\maketitle


\section{Introduction}\label{S:intro}

Let $\DD$ denote the open unit disk, 
and $\hol(\DD)$ denote the set of holomorphic functions on $\DD$.
Given $f\in\hol(\DD)$, we define its Dirichlet integral by
\[
\cD(f):=\frac{1}{\pi}\int_\DD|f'(z)|^2\,dA(z).
\]
The Dirichlet space $\cD$ consists of those $f\in\hol(\DD)$ for which $\cD(f)<\infty$. 
It is easy to see that $\cD$ is contained in the Hardy space $H^2$, and that it becomes a 
Hilbert space when endowed with the norm $\|\cdot\|_\cD$ defined by
\[
\|f\|_\cD^2:=\|f\|_{H^2}^2+\cD(f).
\]
For further information on the Dirichlet space we refer to the book \cite{EKMR14}.

Our main result is the following theorem.

\begin{theorem}\label{T:D}
Let  $\Lambda:\cD\to\CC$ be a linear functional  
such that $\Lambda(1)=1$ and  $\Lambda(g)\ne0$ for all nowhere-vanishing functions $g\in \cD$.
Then there exists $a\in \DD$ such that
$\Lambda(f)=f(a)$ for all $f\in\cD$.
\end{theorem}

This result can be viewed as a Dirichlet-space analogue of the classical
Gleason--Kahane--\.Zelazko (GKZ) theorem for Banach algebras.
As in the original GKZ theorem, continuity of $\Lambda$ is not assumed.
This result is thus an improvement of a theorem obtained in \cite{MR15}, 
where it was necessary to assume continuity of~$\Lambda$.

A consequence of this theorem is the following characterization of weighted composition operators
on $\cD$. Again, no continuity is assumed.

\begin{theorem}\label{T:wco}
Let $T:\cD\to\hol(\DD)$ be a linear map that maps nowhere-vanishing functions to nowhere-vanishing functions.
Then there exist holomorphic functions $\phi:\DD\to\DD$ and $\psi:\DD\to\CC\setminus\{0\}$
such that 
\[
Tf=\psi.(f\circ\phi) \qquad(f\in\cD).
\]
\end{theorem}

We also seek to extend Theorem~\ref{T:D} to certain weighted Dirichlet spaces.
Given a positive superharmonic function $w$ on $\DD$, we define $\cD_w$
to be the set of $f\in\hol(\DD)$ such that
\[
\cD_w(f):=\frac{1}{\pi}\int_\DD |f'(z)|^2w(z)\,dA(z)<\infty.
\]
The weight $w$ is automatically integrable, so $\cD_w$ contains all polynomials.
One can show that $\cD_w\subset H^2$, and that $\cD_w$ becomes a Hilbert space when endowed with
the norm $\|\cdot\|_{\cD_w}$ defined by
\[
\|f\|_{\cD_w}^2:=\|f\|_{H^2}^2+\cD_w(f).
\]

Obviously, the classical Dirichlet space $\cD$ corresponds to taking $w\equiv1$.
Other interesting examples include the standard weighted Dirichlet spaces $\cD_\alpha$ for $0<\alpha<1$
(obtained by taking $w(z):=(1-|z|^2)^\alpha$),
and the harmonically weighted Dirichlet spaces introduced by Richter in \cite{Ri91} 
and further studied by Richter and Sundberg in \cite{RS91}.
The study of Dirichlet spaces with general superharmonic weights 
was initiated by Aleman in his habilitation thesis \cite{Al93},
where further details on this subject may be found.

We prove:

\begin{theorem}\label{T:Dw}
Let $w$ be a positive superharmonic function on $\DD$.
Let  $\Lambda:\cD_w\to\CC$ be a linear functional  
such that $\Lambda(1)=1$ and  $\Lambda(g)\ne0$ for all nowhere-vanishing functions $g\in \cD_w$.
Then there exists $a\in \DD$ such that
$\Lambda(f)=f(a)$ for all $f\in\cD_w$.
\end{theorem}

An issue that arises in the course of the proof  of this theorem is whether $\cD_w$ contains 
surjective functions, namely functions $f$ such that $f(\DD)=\CC$.
The question of which function spaces on $\DD$ contain surjective functions has been extensively
studied, but  these studies date from before the introduction of the spaces $\cD_w$,
so we believe that it is worth recording the following result explicitly.

\begin{theorem}\label{T:surj}
Let $w$ be a positive superharmonic function on $\DD$.
Then $\cD_w$ contains surjective functions if and only if $\inf_{z\in\DD} w(z)=0$.
\end{theorem}

Theorems~\ref{T:D},  \ref{T:wco} and \ref{T:Dw} are  proved in  \S\ref{S:D}. 
The proof of Theorem~\ref{T:surj} is presented in \S\ref{S:surj}.
We conclude   in \S\ref{S:conclusion} with a discussion of the results above and with some open problems.


\section{Proofs of Theorems~\ref{T:D}, \ref{T:wco} and \ref{T:Dw}}\label{S:D}

The proof of Theorem~\ref{T:D} is based upon the following abstract result obtained in \cite{MR15}.

\begin{theorem}[\protect{\cite[Theorem~1.2]{MR15}}]\label{T:module}
Let $A$ be a  complex unital Banach algebra, let $M$ be a left $A$-module,
and let $S$ be a non-empty subset of $M$ satisfying the following conditions:
\begin{enumerate}[(1)]
\item $S$ generates $M$ as an $A$-module;
\item if $a\in A$ is invertible and $s\in S$, then $a s\in S$;
\item for all $s_1,s_2\in S$, there exist $a_1,a_2\in A$ such that 
$a_j S\subset S~(j=1,2)$ and $a_1 s_1=a_2 s_2$.
\end{enumerate}
Let $\Lambda:M\to\CC$ be a linear functional such that 
$\Lambda(s)\ne0$ for all $s\in S$. 
Then there exists a unique character $\chi$ on $A$ such that
\[
\Lambda(a m)=\chi(a)\Lambda(m) \qquad(a\in A,~m\in M).
\]
\end{theorem}

The plan is to apply this theorem,
 taking $M=\cD$ and $S$ to be the set of nowhere-vanishing functions in $\cD$.
Also, we take $A$ to be $\cM(\cD)$, the multiplier algebra of $\cD$, defined by
\begin{align*}
\cM(\cD)&:=\{h\in\hol(\DD):hf\in\cD \text{~for all~}f\in\cD\},\\
\|h\|_{\cM(\cD)}&:=\sup\{\|hf\|_\cD: f\in\cD,~\|f\|_\cD\le1\}.
\end{align*}
One can show that $\cM(\cD)$ is a Banach algebra
and that $\cM(\cD)\subset\cD\cap H^\infty$.
In fact the inclusion is proper, and though there is an exact characterization of elements of $\cM(\cD)$,
it is not easy to use directly. 

Fortunately, it is also possible to approach $\cM(\cD)$ via the theory
of reproducing kernel Hilbert spaces. Aleman, Hartz, McCarthy and Richter \cite{AHMR17}
recently
obtained the following factorization theorem, based on earlier work of
Alpay, Bolotnikov and Kaptano\u glu \cite{ABK02}.

\begin{theorem}[\protect{\cite[Theorem 1]{AHMR17}}]\label{T:AHMR}
Let $\cH$ be a reproducing kernel Hilbert space whose kernel is normalized and has the complete Pick property.
Then, given $f\in\cH$, there exist $h,k$ in the multiplier algebra of $\cH$,
with $k$ nowhere zero, such that $f=h/k$.
\end{theorem}

The terminology is explained in \cite{AHMR17}, and further background may be found in the book \cite{AM02}.
For our purposes, it suffices to remark that (as pointed out in \cite{AHMR17})
the Dirichlet space $\cD$ satisfies the hypotheses
of the theorem, and thus we obtain the following corollary.

\begin{corollary}\label{C:AHMR}
Given $f\in\cD$, there exist $h,k\in\cM(\cD)$, with $k$ nowhere zero on $\DD$, such that $f=h/k$.
\end{corollary}

\begin{proof}[Proof of Theorem~\ref{T:D}]
As proposed earlier, we apply Theorem~\ref{T:module}, taking $M=\cD$ and $A=\cM(\cD)$ and 
 $S$ to be the set of nowhere-vanishing functions in $\cD$. We first need to check that the conditions
 (1), (2) and (3) are satisfied.
 
For condition~(1), it suffices to show that every $f\in\cD$ can be written as $f=g_1+g_2$, where $g_1,g_2\in\cD$ and neither function $g_j$ has a zero in $\DD$. To this end, we remark that, if $f\in\cD$, then the area of its image $f(\DD)$ is bounded above by the Dirichlet integral $\cD(f)$, which is finite. Consequently, we can choose a complex number $\lambda\notin f(\DD)\cup\{0\}$, and then, setting $g_1:=\lambda$ and $g_2:=f-\lambda$, we have the required decomposition.
 
Condition (2) is obviously satisfied, since invertible elements of $\cM(\cD)$ must be everywhere non-zero on $\DD$.

To check condition~(3), let $g_1,g_2$ be nowhere-vanishing elements of $\cD$.
By Corollary~\ref{C:AHMR}, we can write them as $g_j=h_j/k_j$, where $h_1,k_1,h_2,k_2$ are nowhere-vanishing elements of $\cM(\cD)$. Set $a_1:=h_2k_1$ and $a_2:=h_1k_2$. These $a_1,a_2$ are nowhere-vanishing elements of $\cM(\cD)$ and $a_1g_1=a_2g_2$. Thus condition~(3) is satisfied.

By Theorem~\ref{T:module}, 
there exists a character $\chi$ on $\cM(\cD)$ such that
\begin{equation}\label{E:Hplinfun}
\Lambda(hf)=\chi(h)\Lambda(f) \qquad(f\in \cD,~h\in \cM(\cD)).
\end{equation}

Let $a:=\chi(u)$ (where $u$ denotes the function $u(z):=z$).
For all $\lambda\in\CC\setminus\DD$, 
the function $(u-\lambda 1)$ is is non-vanishing in $\DD$, 
so we have  $\Lambda(u-\lambda 1)\ne0$,
whence $\chi(u-\lambda 1)\ne0$ and  $a\ne\lambda$. 
In other words, $a\in\DD$.

To finish the proof, 
we show that $\Lambda(f)=f(a)$ for all $f\in \cD$.
Given $f\in \cD$,  let us define $f_1(z):=(f(z)-f(a))/(z-a)$. 
Then $f_1\in \cD$ and  $f=f(a)1+(u-a1)f_1$. 
Applying $\Lambda$ to both sides of this 
last identity and using \eqref{E:Hplinfun},
we obtain 
\[
\Lambda(f)=f(a)\Lambda(1)+\chi(u-a1)\Lambda(f_1)=f(a)+0,
\]
as desired. This completes the proof of Theorem~\ref{T:D}.
\end{proof}

\begin{proof}[Proof of Theorem~\ref{T:wco}]
Set $\psi:=T(1)$. Clearly $\psi\in\hol(\DD)$ and $\psi(z)\ne0$ for all $z\in\DD$.  
Set $\phi:=T(u)/\psi$, where $u$ is the function  $u(z):=z$. Then also $\phi\in\hol(\DD)$.
Fix $z\in\DD$ and consider the linear functional $\Lambda:\cD\to\CC$ defined by
$\Lambda(f):=(Tf)(z)/\psi(z)$. 
This satisfies the hypotheses of Theorem~\ref{T:D}, so by that theorem there exists $a\in\DD$
such that $\Lambda(f)=f(a)$ for all $f\in\cD$. In particular, taking $f=u$, we see that $\phi(z)=a$.
As this holds for each $z\in\DD$, we conclude that $\phi$ maps $\DD$ into $\DD$, 
and that $Tf(z)=\psi(z)f(\phi(z))$ for all $f\in\cD$ and all $z\in\DD$.
\end{proof}

\begin{proof}[Proof of Theorem~\ref{T:Dw}]
This is nearly the same as the proof of Theorem~\ref{T:D}, but with two differences.

Firstly, we need a new method for  checking condition~(1)
because $\cD_w$, unlike $\cD$, may contain surjective functions
(more on this in the next section).
The following argument was suggested to us by the referee.
Given $f\in\cD_w$, factor it as $f=hg$, where  $h$ is inner and $g$ outer.
By \cite[Chapter~IV, Theorem~3.4]{Al93}, we have $g\in\cD_w$.
Then  $f=(h-1)g+g$ is the sum of two nowhere-vanishing 
functions in $\cD_w$. Thus condition~(1) is verified.

Secondly, in checking condition (3), we need an analogue of Corollary~\ref{C:AHMR}
for the spaces $\cD_w$. This can be proved by combining Theorem~\ref{T:AHMR}
with a theorem of Shimorin \cite{Sh02} asserting that, for every positive superharmonic weight $w$,
the space $\cD_w$ has a complete Pick kernel.
\end{proof}


\section{Proof of Theorem~\ref{T:surj}}\label{S:surj}

In the light of the proof of Theorem~\ref{T:Dw}, 
it is natural to wonder whether $\cD_w$ contains surjective functions.
Theorem~\ref{T:surj}, stated in the introduction, answers this question.
In this section, we prove this theorem.
The proof is based on the following fairly general lemma.

\begin{lemma}\label{L:surj}
Let $X$ be a Banach space of holomorphic functions on $\DD$.
Assume that convergence in the norm of $X$ implies local uniform convergence in $\DD$.
Suppose also that there exist a bounded, non-constant function $h\in X$, vanishing at $0$,
and automorphisms $(\phi_n)_{n\ge1}$ of $\DD$, such that $h\circ\phi_n\in X$ for all $n$ and
$\lim_{n\to\infty}\|h\circ\phi_n\|_X=0$. Then there exists a function $f\in X$ such that $f(\DD)=\CC$.
\end{lemma}

\begin{proof}
By the principle of isolated zeros, there exists $r\in(0,1)$ such that $m:=\min_{|z|=r}|h(z)|>0$.
Set $M:=\sup_{z\in\DD}|h(z)|$ and $\lambda:=2+M/m$.
Replacing $(\phi_n)$ by a subsequence,  we may suppose that, for all $n\ge1$,
\[
\|h\circ\phi_n\|_X<2^{-n}\lambda^{-n}.
\]

For each automorphism $\phi$ of $\DD$, let us write $D_\phi:=\phi^{-1}({\DD_r})$, where $\DD_r:=\{z:|z|<r\}$.
Replacing $(\phi_n)$ by a further subsequence, if necessary, we may suppose that, for all $n\ge2$,
\[
\max\{|h\circ\phi_n(z)|:z\in\overline{D}_{\phi_1}\cup\dots\cup\overline{D}_{\phi_{n-1}}\}<2^{-n}\lambda^{-n}.
\]
This is possible because, since $\|h\circ\phi_n\|_X\to0$, it follows that $h\circ\phi_n\to0$ locally
uniformly on $\DD$.

Define $f:\DD\to\CC$ by
\[
f:=\sum_{n\ge1}\lambda^n(h\circ\phi_n).
\]
Since $\sum_{n\ge1}\|\lambda^n(h\circ\phi_n)\|_X\le \sum_{n\ge1}2^{-n}<\infty$,
the series defining $f$ converges in the norm of $X$, hence also locally uniformly on $\DD$. 
In particular, we have $f\in X$.

We now show that $f(\DD)=\CC$.
Let $w\in\CC$. Since $\lambda>1$ and $m>M/(\lambda-1)$, we may choose $n$ large enough so that
\[
|w|+1<\lambda^n\Bigl(m-\frac{M}{\lambda-1}\Bigr).
\]
Fix this $n$ and set  
\[
F:=\lambda^n(h\circ \phi_n) 
\quad\text{and}\quad
G:=f-\lambda^n(h\circ\phi_n)-w=\sum_{k\ne n}\lambda^k(h\circ\phi_k)-w.
\]
Clearly both $F$ and $G$ are holomorphic on $\DD$. Further, we have
\[
\min_{z\in \partial D_{\phi_n}}|F(z)|
=\min_{z\in \phi_n^{-1}(\partial \DD_r)}\lambda^n |h\circ\phi_n(z)|
=\min_{|w|=r}\lambda^n|h(w)|
=\lambda^n m,
\]
and
\begin{align*}
\max_{z\in \partial D_{\phi_n}}|G(z)|
&\le\max_{z\in \partial D_{\phi_n}}\Bigl(\sum_{k<n}\lambda^k|h\circ\phi_k(z)|+\sum_{k>n}\lambda^k|h\circ\phi_k(z)|+|w|\Bigr)\\
&\le \sum_{k<n}\lambda^k M+\sum_{k>n}2^{-k}+|w|\\
&\le \frac{\lambda^nM}{\lambda-1}+1+|w|.
\end{align*}
By our choice of $n$, it follows that $\max_{\partial D_{\phi_n}}|G|<\min_{\partial D_{\phi_n}}|F|$.
Therefore, by Rouch\'e's theorem, $F$  and $F+G$ have the same number of zeros in $D_{\phi_n}$.
Now $F$ has at least one zero there, since $\phi_n^{-1}(0)\in D_{\phi_n}$ and
\[
F(\phi_n^{-1}(0))=\lambda^n (h\circ\phi_n)(\phi_n^{-1}(0))=\lambda^n h(0)=0.
\]
Therefore $F+G$ has at least one zero in  $D_{\phi_n}$. Since $F+G=f-w$, this implies that $w\in f(D_{\phi_n})$.
In particular $w\in f(\DD)$.
\end{proof}

\begin{proof}[Proof of Theorem~\ref{T:surj}]
The `only if' is easy. Indeed, if $\inf_{z\in\DD}w(z)>0$, then $\cD_w\subset\cD$,
and, as already observed, $\cD$ contains no surjective functions.

We now turn to the `if'. Suppose that $\inf_{z\in\DD}w(z)=0$.
We are going to check that the hypotheses of Lemma~\ref{L:surj} are satisfied.
Clearly $\cD_w$ is a Banach space in which norm convergence implies local uniform convergence.
Since $\inf_{z\in\DD}w(z)=0$,
there exists a sequence $(a_n)$ in $\DD$ such that $w(a_n)\to0$.
Replacing $(a_n)$ by a subsequence, 
we can suppose that $(a_n)$ converges to some limit $a\in\overline{\DD}$.
If $a\in\DD$, then by lower semicontinuity of $w$ we have $w(a)=0$, contradicting positivity of~$w$;
so $a\in\partial \DD$.
Define $h(z):=z(a-z)$ and $\phi_n(z):=(a_n-z)/(1-\overline{a}_nz)$.
Clearly $h$ is bounded and $h(0)=0$. Also
\[
\|h\circ\phi_n\|_{H^2}^2=\|\phi_n(a-\phi_n)\|_{H^2}^2=2-2\Re\langle a,\phi_n\rangle=2-2\Re (a\overline{a}_n)\to0.
\]
Further, we have
\begin{align*}
\cD_w(h\circ\phi_n)
&=\frac{1}{\pi}\int_\DD|(h\circ\phi_n)'(z)|^2w(z)\,dA(z)\\
&=\frac{1}{\pi}\int_\DD |h'(\zeta)|^2(w\circ\phi_n^{-1})(\zeta)\,dA(\zeta)\\
&\le \frac{1}{\pi}\int_\DD 9(w\circ\phi_n^{-1})(\zeta)\,dA(\zeta)\\
&\le 9(w\circ\phi_n^{-1})(0)=9w(a_n)\to0,
\end{align*}
where the final inequality arises from the fact that $w\circ\phi_n^{-1}$ is a superharmonic function on $\DD$.
Hence
\[
\|h\circ\phi_n\|_{\cD_w}^2=\|h\circ\phi_n\|_{H^2}^2+\cD_w(h\circ\phi_n)\to0.
\]
Thus the hypotheses of Lemma~\ref{L:surj} are satisfied, and we deduce that $\cD_w$
contains a function $f$ such that $f(\DD)=\CC$. This completes the proof.
\end{proof}


\section{Concluding remarks}\label{S:conclusion}

There is a version of Theorem~\ref{T:D} for Hardy spaces. The following result was obtained in \cite{MR15}.

\begin{theorem}[\protect{\cite[Theorem~2.1]{MR15}}]\label{T:Hp}
Let  $0<p\le\infty$ and let $\Lambda:H^p\to\CC$ be a linear functional  
such that $\Lambda(1)=1$ and  $\Lambda(g)\ne0$ for all outer functions $g\in H^p$.
Then there exists $a\in \DD$ such that
$\Lambda(f)=f(a)$ for all $f\in H^p$.
\end{theorem}

A comparison of Theorems~\ref{T:Hp} and \ref{T:D} reveals that these results are not exact analogues
of one another. Indeed, in Theorem~\ref{T:D} we suppose that $\Lambda$ is non-zero on nowhere-vanishing functions, whereas in Theorem~\ref{T:Hp} it suffices to assume that $\Lambda$ is non-zero on the (strictly smaller) class of outer functions. Why the difference?

The Hardy-space case is much easier to treat, because the multiplier algebra is exactly equal to $H^\infty$, and there is a satisfactory factorization theory (inner-outer factorization) that makes it easy to check conditions (1)--(3) of Theorem~\ref{T:module}. In particular, it allows us to prove Theorem~\ref{T:Hp} under the weaker outer-function assumption.

In the case of the Dirichlet space, although the outer factor of a function in $\cD$ still belongs to $\cD$, the inner factor need not belong to $\cD$, still less to $\cM(\cD)$. (In fact, the only inner functions that belong to $\cD$ are finite Blaschke products \cite[Corollary~7.6.10]{EKMR14}.) This explains why we need the  factorization result Theorem~\ref{T:AHMR} (a deep theorem, based on the so-called realization formula for spaces with complete Pick kernels), and also why we  resort to the trick of exploiting the fact that the  Dirichlet space contains no surjective functions. To extend out results, it would be helpful to answer some of the following questions, which we believe are of interest in their own right.

\begin{questions}
Let $f\in\cD$.
\begin{enumerate}[(1)]
\item Can we write  $f=\sum_{j=1}^n h_jg_j$, where $h_j\in\cM(\cD)$ and $g_j\in\cD$, and with the $g_j$ being outer functions? 
\item Can we even take the $g_j$ to be cyclic for $\cD$? 
\item Is this possible even with $n=1$?
\item What if we replace $\cD$ by $\cD_w$, where $w$ is a superharmonic weight? 
\end{enumerate}
\end{questions}

Looking beyond the Dirichlet space, it would certainly be of interest to answer the analogous questions for other function spaces too.

\section*{Acknowledgements}
We are grateful to the anonymous referee for suggesting the argument used in
proving Theorem~\ref{T:Dw}.


\end{document}